 \documentclass[11pt,reqno]{amsart}
\usepackage{amsmath,amsthm,amssymb,bm,bbm,dsfont}
\usepackage{mathtools}\mathtoolsset{centercolon}
\usepackage[shortlabels]{enumitem}
\usepackage[mathscr]{eucal}
\usepackage{amsaddr}
\usepackage{upgreek}
\usepackage[left=2cm,right=2cm,top=2cm,bottom=2cm]{geometry}
\mathtoolsset{showonlyrefs=true}
\usepackage{setspace}
\usepackage{graphicx}
\usepackage{verbatim}
\usepackage{float}
\usepackage{placeins}
\usepackage{array, soul}
\usepackage{booktabs}
\usepackage{threeparttable}
\usepackage[update,prepend]{epstopdf}
\usepackage{multirow}
\usepackage{amsfonts,amssymb,dsfont}
\usepackage[abs]{overpic}
\usepackage{xfrac}
\usepackage{enumitem}

\usepackage[usenames,dvipsnames]{color}
\usepackage{colortbl}
\usepackage[hidelinks]{hyperref}
\hypersetup{
	unicode=false,          % non-Latin characters in Acrobat’s bookmarks
	pdftoolbar=true,        % show Acrobat’s toolbar?
	pdfmenubar=true,        % show Acrobat’s menu?
	pdffitwindow=false,     % window fit to page when opened
	pdfstartview={FitH},    % fits the width of the page to the window
	pdftitle={Sharp Tilting Inequalities},    % title
	pdfauthor={Hao-Chung Cheng and Po-Chieh Liu},     % author
	pdfsubject={Sharp trace inequalities and their infinite-dimensional extensions},   % subject of the document
	pdfkeywords={trace inequalities, logarithmic derivative, trace-class operators}, % list of keywords
	pdfnewwindow=true,      % links in new PDF window
	colorlinks=true,        % false: boxed links; true: colored links
	linkcolor=Red,          % color of internal links (change box color with linkbordercolor)
	citecolor=ForestGreen,  % color of links to bibliography
	filecolor=Magenta,      % color of file links
	urlcolor=BlueViolet,    % color of external links
}
\usepackage{doi}
\usepackage{url}
\usepackage{caption, subcaption}
\usepackage{enumitem}
\usepackage{shadethm}

\makeatletter
\ifx\@NODS\undefined%

\let\mathbb=\mathds
\else%
\fi
\makeatother

\DeclareMathOperator{\Tr}{Tr}

\newcommand{\be}{{\mathbf e}}

\def\0{{\mathbf{0}}}
\def\1{{\mathbf{1}}}
\def\2{{\mathbf{2}}}
\def\3{{\mathbf{3}}}
\def\4{{\mathbf{4}}}
\def\5{{\mathbf{5}}}
\def\6{{\mathbf{6}}}

\def\7{{\mathbf{7}}}
\def\8{{\mathbf{8}}}
\def\9{{\mathbf{9}}}

\def\be{\begin{equation}}
\def\ee{\end{equation}}
\def\bea{\begin{eqnarray}}
\def\eea{\end{eqnarray}}

\DeclareMathOperator{\D}{D}
\DeclareMathOperator{\supp}{supp}
\DeclareMathOperator{\ran}{ran}
\newcommand{\dd}{\,\mathrm d}
\newcommand{\ip}[2]{\langle #1,#2\rangle}
\newcommand{\norm}[1]{\lVert #1\rVert}

\theoremstyle{plain}
\newshadetheorem{theo}{Theorem}

\newshadetheorem{prop}[theo]{Proposition} %[section]
\newtheorem{lemm}[theo]{Lemma} %[section]
\newtheorem{coro}[theo]{Corollary} %[section]
\theoremstyle{definition}
\theoremstyle{remark}
\newtheorem{remark}{Remark}[section]

\newcommand{\vertiii}[1]{{\left\vert\kern-0.25ex\left\vert\kern-0.25ex\left\vert #1
		\right\vert\kern-0.25ex\right\vert\kern-0.25ex\right\vert}}

\makeatletter
\newcommand{\opnorm}{\@ifstar\@opnorms\@opnorm}
\newcommand{\@opnorms}[1]{%
	\left|\mkern-1.5mu\left|\mkern-1.5mu\left|
	#1
	\right|\mkern-1.5mu\right|\mkern-1.5mu\right|
}
\newcommand{\@opnorm}[2][]{%
	\mathopen{#1|\mkern-1.5mu#1|\mkern-1.5mu#1|}
	#2
	\mathclose{#1|\mkern-1.5mu#1|\mkern-1.5mu#1|}
}
\makeatother

\begin{document}

%%%%% How to disable amsart.cls to capitalize the article title?
\let\origmaketitle\maketitle
\def\maketitle{
	\begingroup
	\def\uppercasenonmath##1{} % this disables uppercasing title
	\let\MakeUppercase\relax % this disables uppercasing authors
	\origmaketitle
	\endgroup
}
%%%%%%%%%%%%%%%%%%%%%%%%%%%%%%%%

\title{\bfseries \Large{Sharp Tilting Inequalities
		}}

\author{ \normalsize 
{Hao-Chung Cheng}$^{1\text{--}5}$
and
{Po-Chieh Liu}$^{1,2}$
}
\address{\small  	
$^1$Department of Electrical Engineering and Graduate Institute of Communication Engineering,\\ National Taiwan University, Taipei 106, Taiwan (R.O.C.)\\
$^2$Department of Mathematics, National Taiwan University\\
$^3$Center for Quantum Science and Engineering,  National Taiwan University\\
$^4$Hon Hai (Foxconn) Quantum Computing Center, New Taipei City 236, Taiwan (R.O.C.)\\
$^5$Physics/Math Division, National Center for Theoretical Sciences, Taipei 10617, Taiwan (R.O.C.)\\
}

\email{\href{mailto:haochung.ch@gmail.com}{haochung.ch@gmail.com}}

\date{\today}

\begin{abstract} 
For positive semidefinite matrices $A,B$ and $1/2\le\alpha\le1$, we establish
\[
\Tr\!\left[A\,\D\log[A^\alpha+B^\alpha](B^\alpha)\right]
\le s^s(1-s)^{1-s}\Tr[A^\alpha B^{1-\alpha}],
\qquad s=\frac{1-\alpha}{\alpha},
\]
with the optimal coefficient.
The proof uses a superoperator representation of the directional derivative of the logarithm and
an Araki-type estimate for the derivative of a fractional power.
Both this inequality and its symmetric-quotient counterpart extend to positive trace-class $A$
and positive, possibly unbounded, self-adjoint $B$ on a separable Hilbert space.
\end{abstract}

%\subjclass{Primary 15A42, 15A45, 15A60, 47A60}
%
%\keywords{
%Araki--Lieb--Thirrling inequality,
%Golden--Thompson inequality,
%trace inequality
%}

\maketitle

%\tableofcontents

\section{Introduction}

For a real continuously differentiable function $f$ on $(0,\infty)$,
$\D f[X](H)$ denotes its Fr\'echet derivative at $X>0$ in a Hermitian
direction $H$; we abbreviate $\D_s[X]=\D(x^s)[X]$ throughout the paper.
On the matrix space, we denote the Hilbert--Schmidt inner product
$\ip{H}{K}=\Tr[H^*K]$. Logarithms are natural, zero powers denote
support projections, and $s^s(1-s)^{1-s}$ is interpreted as $1$ at
$s=0,1$.

For positive semidefinite matrices $A$ and $B$, and any $\alpha \in [1/2,1]$, $s = \frac{1-\alpha}{\alpha}$, we prove the following inequality (Theorem~\ref{thm:log})
\begin{align} \label{eq:sharp_tilting}
\Tr\!\left[A\,\D\log[A^\alpha+B^\alpha](B^\alpha)\right]
\le s^s(1-s)^{1-s}\Tr[A^\alpha B^{1-\alpha}].
\end{align}
The inequality becomes an equality when $(1-s) A^{\alpha} = s B^{\alpha}$.
This extends the scalar inequality
\begin{equation}\label{eq:scalar-intro}
	\frac{x}{1+x}\le s^s(1-s)^{1-s}x^s,
	\qquad x>0,\quad 0\le s\le1
\end{equation}
to noncommutative matrices and improves the coefficient obtained in the authors' previous work based on the operator layer cake theorem \cite[Proposition 2.3]{ChengLiu}.
The operational bounds relying on the tilting inequality in \cite{ChengLiu} can thereby be improved.

Our technique is to exploit a superoperator representation \cite[(3.12)]{Bathory}, \cite{Bhatia}
\begin{align}\label{eq:log-average-intro}
	\D\log[X](H)
	&=\int_0^1\mathcal L_v(X)^{-1}(H)\dd v,
	\qquad
	\mathcal L_v(X)(H)=(1-v)XH+vHX
\end{align}
for expressing the quotient.
Moreover, by defining the induced superoperator for $X,Y>0$,
\begin{equation}\label{eq:gamma}
	\begin{split}
%		\mathcal L_v(X)(H)&=(1-v)XH+vHX,\\
		\Gamma_{v,Y\mid X}
		&\coloneqq\mathcal L_v(X)^{-1/2}\mathcal L_v(Y)\mathcal L_v(X)^{-1/2},
		\qquad 0\le v\le1,
	\end{split}
\end{equation}
a key intermediate term can be written as 
\begin{align}
	\Tr A-
	\Tr\!\left[A^{(1+\alpha)/2}\D\log[A^\alpha+B^\alpha]
	(A^{(1+\alpha)/2})\right]
	&=\int_0^1\ip{\sqrt A}{
		\Gamma_{v,B^\alpha\mid A^\alpha}
		(I+\Gamma_{v,B^\alpha\mid A^\alpha})^{-1}(\sqrt A)}\dd v.
\end{align}
%\begin{align}
%	\D\log[X+Y](Y) = \int_0^1 \Gamma_{v,Y\mid X} \left( I + \Gamma_{v,Y\mid X} \right)^{-1} (I) \dd v.
%\end{align}
The superoperator in the integrand hence mimics the role of $\frac{x}{1+x}$ in the scalar inequality \eqref{eq:scalar-intro}.\footnote{
Indeed, using \eqref{eq:gamma}, we can express Umegaki relative entropy for density operators $\rho$ and $\sigma$ with $\mathrm{supp}(\rho)\subseteq\mathrm{supp}(\sigma)$
as  $D(\rho\|\sigma) := \Tr[\rho (\log \rho - \log \sigma)]
= \int_0^1 \left\langle \sqrt{\sigma}, \Gamma_{v,\rho\mid \sigma} \log (\Gamma_{v,\rho\mid \sigma})(\sqrt{\sigma})\right\rangle\dd v$, via Frenkel's formula \cite{Frenkel, Frenkel_equivalent}.
Similarly, the \emph{integral $f$-divergence} in \cite{f-divergence, some_properties, LiuHircheCheng} admits the formula
$D_f(\rho\|\sigma) = \int_0^1 \left\langle \sqrt{\sigma}, f(\Gamma_{v,\rho\mid \sigma})(\sqrt{\sigma})\right\rangle\dd v$ for some proper convex function $f$.
}
This then allows us to \emph{lift} the scalar inequality to a superoperator representation.

In Section~\ref{sec:power_estimate}, we first provide necessary intermediate inequalities for proving \eqref{eq:sharp_tilting}.
We present the main proof of the sharp tilting inequality in Section~\ref{sec:log}.
In Section~\ref{sec:symmetric}, we show that the same tilting inequality holds if we replace the logarithm derivative by the symmetric sandwiched quotient.
Section~\ref{sec:infinite} extends both inequalities to a separable Hilbert space: Section~\ref{sec:trace-class} treats trace-class $A$ and bounded $B$, and Section~\ref{sec:unbounded} allows unbounded $B$.

\section{An averaged power estimate} \label{sec:power_estimate}

We first state two derivative facts. Their proofs are deferred to
Appendices~\ref{app:identity} and~\ref{app:derivative}.

\begin{lemm}[Averaged derivative identity]\label{lem:identity}
	Let $A,B>0$, $r\in\mathbb R$, and $f\in C^1(0,\infty)$. For $u\ge0$,
	\begin{align}
		&\frac{\mathrm d}{\mathrm du}\int_0^1
		\ip{A^{(1+r)/2}}{f(\Gamma_{v,B+uI\mid A})(A^{(1+r)/2})}\dd v
		\notag\\
		&\qquad=\Tr\!\left[A^{r/2}
		\D f\!\left[A^{-1/2}(B+uI)A^{-1/2}\right](A^{r/2})\right].
		\label{eq:derivative-identity}
	\end{align}
	At $u=0$, the derivative may be taken from the right.
\end{lemm}

\begin{lemm}[Araki-type estimate for a power derivative]\label{lem:derivative}
	For $A,B>0$ and $0<s<1$,
	\begin{equation}\label{eq:derivative}
		\Tr\!\left[A^{s/2}\D_s[B](A^{s/2})\right]
		\ge s\Tr\!\left[A(A^{1/2}BA^{1/2})^{s-1}\right].
	\end{equation}
\end{lemm}

\begin{lemm}[Averaged power estimate]\label{lem:moment}
	For $A,B>0$, $\alpha \in [1/2,1]$, and $s = \frac{1-\alpha}{\alpha}$,
	\begin{equation}\label{eq:moment}
		\int_0^1\ip{\sqrt A}{\Gamma_{v,B^\alpha\mid A^\alpha}^{s}(\sqrt A)}\dd v
		\le\Tr[A^\alpha B^{1-\alpha}].
	\end{equation}
	Equality holds at $\alpha=1/2$ and $\alpha=1$.
\end{lemm}

\begin{proof}
	Suppose first that $1/2<\alpha<1$, so $0<s<1$. For $u\ge0$, set
	\[
	g(u)=\int_0^1
	\ip{\sqrt A}{\Gamma_{v,B^\alpha+uI\mid A^\alpha}^{s}(\sqrt A)}\dd v
	-\Tr\!\left[A^\alpha(B^\alpha+uI)^s\right].
	\]
	Since $\alpha(1+s)=1$, Lemma~\ref{lem:identity}, with
	$A^\alpha,B^\alpha$ in place of $A,B$, gives
	\begin{align*}
		g'(u)
		&=\Tr\!\left[A^{(1-\alpha)/2}
		\D_s\!\left[A^{-\alpha/2}(B^\alpha+uI)A^{-\alpha/2}\right]
		(A^{(1-\alpha)/2})\right]\\
		&\qquad-s\Tr\!\left[A^\alpha(B^\alpha+uI)^{s-1}\right]
		\ge0.
	\end{align*}
	The last inequality is Lemma~\ref{lem:derivative}, applied to
	$A^\alpha$ and $A^{-\alpha/2}(B^\alpha+uI)A^{-\alpha/2}$.
	
	To determine the sign of $g$, note that
	\[
	u\mathcal L_v(A^\alpha)^{-1}
	\le\Gamma_{v,B^\alpha+uI\mid A^\alpha}
	\le(u+\norm{B^\alpha})\mathcal L_v(A^\alpha)^{-1}.
	\]
	By the L\"owner--Heinz inequality \cite{Bhatia} and the identity
	\[
	\ip{\sqrt A}{\mathcal L_v(A^\alpha)^{-s}(\sqrt A)}=\Tr A^\alpha,
	\]
	both terms defining $g(u)$ lie between $u^s\Tr A^\alpha$ and
	$(u+\norm{B^\alpha})^s\Tr A^\alpha$. Consequently,
	\[
	|g(u)|\le\bigl((u+\norm{B^\alpha})^s-u^s\bigr)\Tr A^\alpha\longrightarrow0
	\qquad (u\to\infty).
	\]
	Thus $g$ is increasing with limit zero, and $g(0)\le0$ proves
	\eqref{eq:moment}.
	
	At $\alpha=1$, both sides are $\Tr A$. At $\alpha=1/2$, the integrand is
	\[
	\ip{A^{1/4}}{\mathcal L_v(B^{1/2})(A^{1/4})}
	=\Tr[A^{1/2}B^{1/2}],
	\]
	which proves the other endpoint.
\end{proof}

\section{The sharp logarithmic trace inequality} \label{sec:log}

For $X>0$, the logarithmic derivative has the representations
\begin{equation}\label{eq:log-average}
	\D\log[X](H)=\int_0^\infty(X+uI)^{-1}H(X+uI)^{-1}\dd u
	=\int_0^1\mathcal L_v(X)^{-1}(H)\dd v;
\end{equation}
see \cite{Bathory,Bhatia,Higham}. In an eigenbasis of $X$ with eigenvalues $(x_i)_i$, both maps multiply
the $(i,j)$ entry of $H$ by
\[
\begin{cases}
	\dfrac{\log x_i-\log x_j}{x_i-x_j},&x_i\ne x_j,\\[4pt]
	1/x_i,&x_i=x_j.
\end{cases}
\]
In particular, $\D\log[X]$ is self-adjoint for the Hilbert--Schmidt
inner product, preserves positivity, and satisfies $\D\log[X](X)=I$.

We also need the elementary inequality
\begin{equation}\label{eq:log-cs}
	\Tr\!\left[A^p\D\log[B](A^q)\right]
	\ge\Tr\!\left[A^{(p+q)/2}\D\log[B](A^{(p+q)/2})\right],
	\qquad p,q\in\mathbb R,
\end{equation}
for $A,B>0$. Indeed, write $A=\sum_i a_iP_i$. The coefficients
$k_{ij}=\Tr[P_i\D\log[B](P_j)]$ are nonnegative and symmetric in $i,j$.
Therefore \eqref{eq:log-cs} follows by summing
\[
\frac12\bigl(a_i^p a_j^q+a_i^q a_j^p\bigr)
\ge(a_i a_j)^{(p+q)/2}
\]
against $k_{ij}$.

\begin{theo}[Sharp tilting inequality]\label{thm:log}
	Let $A,B\ge0$ be matrices and let $1/2\le\alpha\le1$, $s=\frac{1-\alpha}{\alpha}$. Then
	\begin{equation}\label{eq:log-main}
		\Tr\!\left[A\,\D\log[A^\alpha+B^\alpha](B^\alpha)\right]
		\le s^s(1-s)^{1-s}\Tr[A^\alpha B^{1-\alpha}].
	\end{equation}
	The logarithmic derivative is evaluated on $\supp(A+B)$ and extended
	by zero on its orthogonal complement; at $\alpha=1$, $B^0$ is the
	support projection of $B$. For $A,B>0$, the inequality has the refinement
	\begin{align}
		\Tr\!\left[A\,\D\log[A^\alpha+B^\alpha](B^\alpha)\right]\notag
		&\le s^s(1-s)^{1-s}
		\int_0^1\ip{\sqrt A}{\Gamma_{v,B^\alpha\mid A^\alpha}^{s}(\sqrt A)}\dd v
		\notag\\
		&\le s^s(1-s)^{1-s}\Tr[A^\alpha B^{1-\alpha}].
		\label{eq:log-chain}
	\end{align}
	The coefficient is optimal for every $\alpha$. If $A,B>0$ and
	$1/2<\alpha<1$, equality in \eqref{eq:log-main} holds if and only if
	\begin{equation}\label{eq:equality}
		B^\alpha=\frac{1-s}{s}A^\alpha.
	\end{equation}
\end{theo}

\begin{proof}
	First assume $A,B>0$. The identity
	$\D\log[A^\alpha+B^\alpha](A^\alpha+B^\alpha)=I$ and
	\eqref{eq:log-cs} give
	\begin{align}
		\Tr\!\left[A\,\D\log[A^\alpha+B^\alpha](B^\alpha)\right]
		&=\Tr A-
		\Tr\!\left[A\,\D\log[A^\alpha+B^\alpha](A^\alpha)\right]\notag\\
		&\le\Tr A-
		\Tr\!\left[A^{(1+\alpha)/2}\D\log[A^\alpha+B^\alpha]
		(A^{(1+\alpha)/2})\right]\notag\\
		&=\int_0^1\ip{\sqrt A}{
			\Gamma_{v,B^\alpha\mid A^\alpha}
			(I+\Gamma_{v,B^\alpha\mid A^\alpha})^{-1}(\sqrt A)}\dd v.
		\label{eq:common-bound}
	\end{align}
	For the last equality, use \eqref{eq:log-average},
	$\mathcal L_v(A^\alpha)^{1/2}(\sqrt A)=A^{(1+\alpha)/2}$, and
	\[
	(I+\Gamma_{v,B^\alpha\mid A^\alpha})^{-1}
	=\mathcal L_v(A^\alpha)^{1/2}
	\mathcal L_v(A^\alpha+B^\alpha)^{-1}\mathcal L_v(A^\alpha)^{1/2}.
	\]
	If $0<s<1$, the maximum of $x^{1-s}/(1+x)$ occurs at
	$x=(1-s)/s$ and equals $s^s(1-s)^{1-s}$. The endpoint bounds are
	immediate. Applying \eqref{eq:scalar-intro} to each $\Gamma$ in
	\eqref{eq:common-bound}, and then using Lemma~\ref{lem:moment}, proves
	\eqref{eq:log-chain}.
	
	For singular $A$ or $B$ and $\alpha<1$, restrict to
	$\supp(A+B)$, where $A^\alpha+B^\alpha>0$. Apply the result to
	$A+\varepsilon I$ and $B+\varepsilon I$ on that space, and let
	$\varepsilon\downarrow0$. Continuity of powers and of the logarithmic
	derivative at a positive definite matrix proves \eqref{eq:log-main}.
	At $\alpha=1$, take $\alpha\uparrow1$ in this finite-dimensional
	inequality. The left side converges on $\supp(A+B)$, whereas
	$B^{1-\alpha}\to B^0$. The common kernel contributes zero.
	
	For $0<s<1$, condition \eqref{eq:equality} gives equality by simultaneous
	diagonalization. Conversely, suppose $A,B>0$ and equality holds.
	The difference between the right and left sides of
	\eqref{eq:scalar-intro} is nonnegative and vanishes only at
	$x=(1-s)/s$. Equality in the scalar step of the proof forces the integral
	of the corresponding nonnegative quadratic form to vanish. This
	quadratic form is continuous in $v$, so it vanishes at $v=0$.
	Spectral calculus therefore gives
	\[
	\left(\Gamma_{0,B^\alpha\mid A^\alpha}-\frac{1-s}{s}I\right)(\sqrt A)=0.
	\]
	Since $\Gamma_{0,B^\alpha\mid A^\alpha}$ is left multiplication by
	$A^{-\alpha/2}B^\alpha A^{-\alpha/2}$ and $\sqrt A$ is invertible,
	\eqref{eq:equality} follows.
	
	Finally, scalar matrices satisfying \eqref{eq:equality} establish
	optimality in the interior. With $A=1$ and $B=b>0$, the ratio of the
	left side of \eqref{eq:log-main} to $\Tr[A^\alpha B^{1-\alpha}]$ is
%	\[
	$\frac{b^{2\alpha-1}}{1+b^\alpha}$.
%	\]
	This tends to $1$ as $b\downarrow0$ at $\alpha=1/2$, and as
	$b\to\infty$ at $\alpha=1$.
\end{proof}

\section{The symmetric quotient}\label{sec:symmetric}

The symmetric quotient admits the same optimal estimate, using the
common upper bound \eqref{eq:common-bound}.

\begin{prop}\label{cor:symmetric}
	For matrices $A,B\ge0$ and $1/2\le\alpha\le1$,
	\begin{equation}\label{eq:symmetric}
		\Tr\!\left[A(A^\alpha+B^\alpha)^{-1/2}
		B^\alpha(A^\alpha+B^\alpha)^{-1/2}\right]
		\le s^s(1-s)^{1-s}\Tr[A^\alpha B^{1-\alpha}].
	\end{equation}
	The inverse is taken on $\supp(A+B)$ and set to zero on the common
	kernel. The coefficient is optimal. For positive definite $A,B$ and
	$1/2<\alpha<1$, equality holds exactly under \eqref{eq:equality}.
\end{prop}

\begin{proof}
	Assume first $A,B>0$. For every Hermitian $H$, diagonalizing $A$ and
	pairing its $(i,j)$ and $(j,i)$ contributions gives
	\begin{equation}\label{eq:pairing}
		\Tr[AHA^\alpha H]
		\ge\Tr[A^{(1+\alpha)/2}HA^{(1+\alpha)/2}H].
	\end{equation}
	Also, the scalar logarithmic mean dominates the geometric mean
	\cite{Bhatia}, \cite[Proposition 2.1]{ChengLiu}, or equivalently
	\[
	\frac{\log x-\log y}{x-y}\le\frac1{\sqrt{xy}},\qquad x,y>0,
	\]
	with the quotient interpreted continuously at $x=y$. Consequently,
	diagonalization of $B>0$ gives
	\begin{equation}\label{eq:log-geom}
		\Tr[H\D\log[B](H)]\le\Tr[HB^{-1/2}HB^{-1/2}],
		\qquad H=H^*.
	\end{equation}
	Apply \eqref{eq:pairing} with $H=(A^\alpha+B^\alpha)^{-1/2}$ and
	\eqref{eq:log-geom} with $H=A^{(1+\alpha)/2}$ and
	$A^\alpha+B^\alpha$ in place of $B$. This yields
	\begin{align*}
		&\Tr\!\left[A(A^\alpha+B^\alpha)^{-1/2}
		A^\alpha(A^\alpha+B^\alpha)^{-1/2}\right]\\
		&\quad\ge\Tr\!\left[A^{(1+\alpha)/2}(A^\alpha+B^\alpha)^{-1/2}
		A^{(1+\alpha)/2}(A^\alpha+B^\alpha)^{-1/2}\right]\\
		&\quad\ge\Tr\!\left[A^{(1+\alpha)/2}\D\log[A^\alpha+B^\alpha]
		(A^{(1+\alpha)/2})\right].
	\end{align*}
	Subtracting from $\Tr A$ and using \eqref{eq:common-bound} gives
	\begin{align*}
		&\Tr\!\left[A(A^\alpha+B^\alpha)^{-1/2}
		B^\alpha(A^\alpha+B^\alpha)^{-1/2}\right]\\
		&\quad\le\int_0^1\ip{\sqrt A}{
			\Gamma_{v,B^\alpha\mid A^\alpha}
			(I+\Gamma_{v,B^\alpha\mid A^\alpha})^{-1}(\sqrt A)}\dd v\\
		&\quad\le s^s(1-s)^{1-s}
		\int_0^1\ip{\sqrt A}{\Gamma_{v,B^\alpha\mid A^\alpha}^{s}(\sqrt A)}\dd v\\
		&\quad\le s^s(1-s)^{1-s}\Tr[A^\alpha B^{1-\alpha}],
	\end{align*}
	by \eqref{eq:scalar-intro} and Lemma~\ref{lem:moment}. The support
	restriction, regularization, and endpoint limit used in
	Theorem~\ref{thm:log} prove the singular case. The same scalar examples
	prove sharpness. In the interior, equality for positive definite
	matrices forces equality in the same scalar step, so the argument in
	Theorem~\ref{thm:log} proves \eqref{eq:equality}; its sufficiency is immediate.
\end{proof}

\section{Extension to a separable Hilbert space}\label{sec:infinite}

Let $\mathcal H$ be a separable complex Hilbert space. We use the
spectral calculus and trace-ideal facts in \cite{ReedSimon,Simon};
$\norm{\,\cdot\,}_2$ denotes the Hilbert--Schmidt norm, with value
$+\infty$ outside that class.

\subsection{Trace-class operators}\label{sec:trace-class}

Let $A$ be a positive trace-class operator and $B$ a bounded positive
operator. For $1/2\le\alpha\le1$, define the logarithmic term by the
strong improper integral
\begin{equation}\label{eq:extended-log}
\D\log[A^\alpha+B^\alpha](B^\alpha)
:=\int_0^\infty(A^\alpha+B^\alpha+tI)^{-1}
B^\alpha(A^\alpha+B^\alpha+tI)^{-1}\dd t.
\end{equation}
The integrals over compact subintervals increase as the intervals expand.
Replacing $B^\alpha$ in the integrand by $A^\alpha+B^\alpha$ gives,
by the spectral theorem, the support projection $(A^\alpha+B^\alpha)^0$.
Since $0\le B^\alpha\le A^\alpha+B^\alpha$, the limit exists and
\begin{equation}\label{eq:contraction}
0\le\D\log[A^\alpha+B^\alpha](B^\alpha)
\le(A^\alpha+B^\alpha)^0\le I.
\end{equation}
This agrees with the usual derivative when the denominator is boundedly
invertible; otherwise \eqref{eq:extended-log} defines the notation and
does not assert Fr\'echet differentiability at a singular operator.

For the symmetric quotient, the inequality
$\norm{B^{\alpha/2}h}\le\norm{(A^\alpha+B^\alpha)^{1/2}h}$
defines $B^{\alpha/2}(A^\alpha+B^\alpha)^{-1/2}$ on the range of
$(A^\alpha+B^\alpha)^{1/2}$ by
\begin{equation}\label{eq:symmetric-factor}
\bigl[B^{\alpha/2}(A^\alpha+B^\alpha)^{-1/2}\bigr]
(A^\alpha+B^\alpha)^{1/2}h=B^{\alpha/2}h.
\end{equation}
The product extends continuously to the closure of that range and is
set to zero on its orthogonal complement; its operator norm is at most one.
The symmetric quotient means the adjoint of this extended product
times the product itself. In particular,
\begin{equation}\label{eq:extended-symmetric}
\Tr\!\left[A(A^\alpha+B^\alpha)^{-1/2}B^\alpha
(A^\alpha+B^\alpha)^{-1/2}\right]
=\norm{B^{\alpha/2}(A^\alpha+B^\alpha)^{-1/2}A^{1/2}}_2^2.
\end{equation}
This agrees with Proposition~\ref{cor:symmetric}. Both quotients are
positive contractions, so their pairings with $A$ are at most $\Tr A$.

The mixed trace is the extended nonnegative trace
\begin{equation}\label{eq:mixed-trace}
\Tr[A^\alpha B^{1-\alpha}]
:=\Tr[A^{\alpha/2}B^{1-\alpha}A^{\alpha/2}]
=\norm{B^{(1-\alpha)/2}A^{\alpha/2}}_2^2\in[0,\infty],
\end{equation}
where $B^0$ is the support projection at $\alpha=1$.
Positivity and monotone convergence also give
\begin{equation}\label{eq:trace-integral}
\Tr\!\left[A\,\D\log[A^\alpha+B^\alpha](B^\alpha)\right]
=\int_0^\infty
\norm{B^{\alpha/2}(A^\alpha+B^\alpha+tI)^{-1}A^{1/2}}_2^2\dd t.
\end{equation}

\begin{theo}[Trace-class extensions]\label{thm:infinite}
Let $A$ be positive trace class and $B$ bounded and positive on
$\mathcal H$. Both \eqref{eq:log-main} and \eqref{eq:symmetric} hold
for $1/2\le\alpha\le1$, with the interpretations
\eqref{eq:extended-log}, \eqref{eq:extended-symmetric}, and
\eqref{eq:mixed-trace}. At $\alpha=1$, their common right side is
$\Tr[AB^0]$. Both constants are optimal on every nonzero
separable Hilbert space.
\end{theo}

\begin{proof}
First let $A$ have finite rank. Choose increasing finite-rank projections
$E_m\to I$ strongly, commuting with $B^0$ and containing $\ran A$.
Such projections are obtained by including $B^0\ran A$ and
$(I-B^0)\ran A$, and enlarging separately in $B^0\mathcal H$ and
$(I-B^0)\mathcal H$. Set $B_m=E_mBE_m$, extended by zero on
$(E_m\mathcal H)^\perp$. Apply Theorem~\ref{thm:log} and
Proposition~\ref{cor:symmetric} on $E_m\mathcal H$.

The operators $B_m$ converge strongly to $B$ and are uniformly bounded;
the same holds for their positive powers and the corresponding resolvents.
The operators inside the norms in \eqref{eq:trace-integral} therefore
converge in Hilbert--Schmidt norm: multiplication on the right by the
fixed finite-rank $A^{1/2}$ turns bounded strong convergence into
Hilbert--Schmidt convergence. Fatou's lemma bounds the logarithmic
trace for $A,B$ by the lower limit of those for $A,B_m$.

For the symmetric quotient, \eqref{eq:symmetric-factor}, the norm
bound by one, and strong convergence of the square roots give
\[
B_m^{\alpha/2}(A^\alpha+B_m^\alpha)^{-1/2}
(A^\alpha+B^\alpha)^{1/2}h\longrightarrow B^{\alpha/2}h.
\]
The range of $(A^\alpha+B^\alpha)^{1/2}$ is dense in
$\supp(A+B)$. Thus the extended products
$B_m^{\alpha/2}(A^\alpha+B_m^\alpha)^{-1/2}$ converge strongly there
to $B^{\alpha/2}(A^\alpha+B^\alpha)^{-1/2}$.
Multiplication on the right by the finite-rank $A^{1/2}$ gives
Hilbert--Schmidt convergence, hence convergence of the symmetric traces.
For $\alpha<1$, finite rank of $A^\alpha$ gives convergence of the
mixed traces as well. At $\alpha=1$,
$\ker(B_m|_{E_m\mathcal H})=E_m\mathcal H\cap\ker B$ by positivity,
so $B_m^0=E_mB^0$ and $\Tr[AB_m^0]=\Tr[AB^0]$.
Both finite-rank bounds follow.

For general $A$, write its nonzero spectral part as
$A=\sum_j a_j\ip{e_j}{\,\cdot\,}e_j$ and set
$A_n=\sum_{j\le n}a_j\ip{e_j}{\,\cdot\,}e_j$. Then
\[
\norm{A_n^\alpha-A^\alpha}\to0,
\qquad \norm{A_n^{1/2}-A^{1/2}}_2\to0.
\]
The operators inside the norms in \eqref{eq:trace-integral} again converge
in Hilbert--Schmidt norm, so Fatou's lemma applies to the logarithmic
trace. The preceding argument on the dense range also gives strong
convergence of $B^{\alpha/2}(A_n^\alpha+B^\alpha)^{-1/2}$ to
$B^{\alpha/2}(A^\alpha+B^\alpha)^{-1/2}$ on $\supp(A+B)$.
Their norms are at most one, so finite-rank approximation of $A^{1/2}$
and $\norm{A_n^{1/2}-A^{1/2}}_2\to0$ give
\[
B^{\alpha/2}(A_n^\alpha+B^\alpha)^{-1/2}A_n^{1/2}
\longrightarrow B^{\alpha/2}(A^\alpha+B^\alpha)^{-1/2}A^{1/2}
\quad\text{in Hilbert--Schmidt norm}.
\]
This proves convergence of the symmetric traces. Finally,
\[
\Tr[A_n^\alpha B^{1-\alpha}]
=\sum_{j\le n}a_j^\alpha\ip{e_j}{B^{1-\alpha}e_j}
\uparrow\Tr[A^\alpha B^{1-\alpha}],
\]
also at $\alpha=1$ with $B^0$ the support projection.
Passing to the limit proves both inequalities, including an infinite
right side. Taking $A=aP$, $B=bP$ for a rank-one projection $P$ reduces
both inequalities to the scalar examples in Theorem~\ref{thm:log},
which prove sharpness.
\end{proof}

\begin{remark}[Finiteness and the endpoint]
For $\alpha<1$, the example
$A=\operatorname{diag}(j^{-1/\alpha})$, $B=I$ on $\ell^2(\mathbb N)$
has $\Tr A<\infty$ but $\Tr A^\alpha=\infty$.
If $A,B$ are both trace class, H\"older's inequality \cite{Simon} gives
$\Tr[A^\alpha B^{1-\alpha}]
\le(\Tr A)^\alpha(\Tr B)^{1-\alpha}$ for $1/2\le\alpha<1$.
Both endpoint bounds above are proved directly, without taking limits
of mixed traces that may be infinite for every $\alpha<1$.
\end{remark}

\subsection{Unbounded positive self-adjoint operators}\label{sec:unbounded}

Let $A$ be positive trace class and $B\ge0$ self-adjoint and densely
defined on $\mathcal H$. Set $B_N=\min(B,NI)$ by spectral calculus.
For $\alpha<1$, interpret the mixed trace as the quadratic-form trace
\begin{equation}\label{eq:unbounded-trace}
\Tr[A^\alpha B^{1-\alpha}]
:=\sup_{N\ge1}\Tr[A^{\alpha/2}B_N^{1-\alpha}A^{\alpha/2}],
\end{equation}
and retain $\Tr[AB^0]$ at $\alpha=1$.
The operator $A^\alpha+B^\alpha$ is positive self-adjoint on the
domain of $B^\alpha$; its quadratic form is defined on the domain of
$B^{\alpha/2}$.
In \eqref{eq:extended-log}, interpret the integrand as
\[
\bigl(B^{\alpha/2}(A^\alpha+B^\alpha+tI)^{-1}\bigr)^*
\bigl(B^{\alpha/2}(A^\alpha+B^\alpha+tI)^{-1}\bigr).
\]
Since $B^\alpha\le A^\alpha+B^\alpha$ in the form sense, this is
bounded, and \eqref{eq:contraction} and \eqref{eq:trace-integral} remain
valid. The same form inequality gives \eqref{eq:symmetric-factor}
for every $h$ in the domain of $B^{\alpha/2}$ and extends
$B^{\alpha/2}(A^\alpha+B^\alpha)^{-1/2}$ continuously to the support,
with value zero on its orthogonal complement. Thus
\eqref{eq:extended-symmetric} also applies, and both traces are at most $\Tr A$.

\begin{coro}[Unbounded extensions]\label{cor:unbounded}
For positive trace-class $A$ and positive self-adjoint $B$ on
$\mathcal H$, both \eqref{eq:log-main} and \eqref{eq:symmetric} hold
for $1/2\le\alpha\le1$, with the interpretations above and the same
optimal constants.
\end{coro}

\begin{proof}
For every $t>0$,
\[
(A^\alpha+B_N^\alpha+tI)^{-1}
\longrightarrow(A^\alpha+B^\alpha+tI)^{-1}\quad\text{strongly}.
\]
Indeed, spectral calculus gives strong convergence of
$(B_N^\alpha+tI)^{-1}$, and a Neumann expansion handles the bounded
perturbation $A^\alpha$ for $t>\norm{A^\alpha}$.
The resolvent identity then gives convergence for all $t>0$;
see \cite[Section~VIII.7]{ReedSimon}.
The identity
\begin{align*}
&(A^\alpha+B_N^\alpha+tI)^{-1}B_N^\alpha(A^\alpha+B_N^\alpha+tI)^{-1}\\
&\quad=(A^\alpha+B_N^\alpha+tI)^{-1}
-t(A^\alpha+B_N^\alpha+tI)^{-2}\\
&\qquad\quad-(A^\alpha+B_N^\alpha+tI)^{-1}A^\alpha
(A^\alpha+B_N^\alpha+tI)^{-1}
\end{align*}
gives strong convergence of the positive integrands. They are uniformly
bounded for each $t$, so their pairings with trace-class $A$ converge.
Fatou's lemma therefore bounds the logarithmic trace by the lower
limit of the corresponding traces with $B_N$.

For the symmetric quotient, let $h$ belong to the domain of $B^\alpha$.
Spectral calculus gives $B_N^{\alpha/2}h\to B^{\alpha/2}h$. Also,
\begin{equation}\label{eq:root-core}
(A^\alpha+B_N^\alpha)^{1/2}h
\longrightarrow(A^\alpha+B^\alpha)^{1/2}h.
\end{equation}
To justify this, truncate $x^{1/2}$ to $\min(x^{1/2},\eta^{1/2})$ for some $\eta>0$.
Strong resolvent convergence gives convergence for the bounded
truncations \cite[Theorem~VIII.20(b), p.~286]{ReedSimon}.
The discarded part on $h$ has norm at most
$\eta^{-1/2}\norm{(A^\alpha+B_N^\alpha)h}$, uniformly tending to zero
as $\eta\to\infty$, because
$\norm{(A^\alpha+B_N^\alpha)h}
\le\norm{A^\alpha}\norm h+\norm{B^\alpha h}$.
The same estimate applies to the limit operator.

Equations~\eqref{eq:symmetric-factor} and \eqref{eq:root-core}, together
with the norm bound by one, give strong convergence of
$B_N^{\alpha/2}(A^\alpha+B_N^\alpha)^{-1/2}$ to
$B^{\alpha/2}(A^\alpha+B^\alpha)^{-1/2}$ on vectors of the form
$(A^\alpha+B^\alpha)^{1/2}h$, with $h$ in the domain of $B^\alpha$.
These vectors are dense in the support by spectral cutoffs for
$A^\alpha+B^\alpha$. The convergence therefore holds throughout the
support. Finite-rank approximation of $A^{1/2}$ gives
Hilbert--Schmidt convergence after multiplication by $A^{1/2}$,
and hence convergence of the symmetric traces.
Apply Theorem~\ref{thm:infinite} to $A,B_N$ and pass to the limit:
the mixed traces increase to \eqref{eq:unbounded-trace} for $\alpha<1$,
and $B_N^0=B^0$ at $\alpha=1$.
Rank-one examples again establish both optimal constants.
\end{proof}

\appendix
\section{Proof of the averaged derivative identity}\label{app:identity}

Fix $u_0\ge0$ and a compact neighborhood on which $B+uI>0$.
The spectra of
$\Gamma_{v,B+uI\mid A}$ and $A^{-1/2}(B+uI)A^{-1/2}$ then lie in one
compact interval $J\subset(0,\infty)$, uniformly in $u$ and $v\in[0,1]$.
Indeed, the lower and upper spectral bounds for either operator are
$\lambda_{\min}(B+uI)/\norm A$ and
$\norm{B+uI}/\lambda_{\min}(A)$.

The Dalecki\u\i--Kre\u\i n formula
\cite[Theorem~5.3.1, p.~154]{Bhatia}, \cite[Chapter~3]{Higham}, in an
eigenbasis of a Hermitian $X$ with spectrum in $J$, is
\begin{equation}\label{eq:DK}
\bigl(\D f[X](H)\bigr)_{ij}=f^{[1]}(x_i,x_j)H_{ij},
\qquad
f^{[1]}(x,y)=
\begin{cases}
\dfrac{f(x)-f(y)}{x-y},&x\ne y,\\[3pt]
f'(x),&x=y.
\end{cases}
\end{equation}
Consequently, on either finite-dimensional space used here,
\begin{equation}\label{eq:DK-bound}
\norm{\D f[X](H)}_2\le\sup_J|f'|\,\norm H_2.
\end{equation}
Uniform polynomial approximation of $f'$ and \eqref{eq:DK-bound} also
show that $\D f[X]$ depends continuously on $X$. Since
$\partial_u\Gamma_{v,B+uI\mid A}=\mathcal L_v(A)^{-1}$ is uniformly
bounded in $v$, the chain rule and \eqref{eq:DK-bound} justify
differentiation under the $v$ integral. Both sides of
\eqref{eq:derivative-identity}, at fixed $u$, are therefore continuous
linear functionals of $f$ in the $C^1(J)$ norm.

We verify their equality first for $f(x)=(x+z)^{-1}$, $z>0$.
The identities
\[
(\Gamma_{v,B+uI\mid A}+zI)^{-1}
=\mathcal L_v(A)^{1/2}\mathcal L_v(B+uI+zA)^{-1}\mathcal L_v(A)^{1/2},
\qquad
\mathcal L_v(A)^{1/2}(A^{(1+r)/2})=A^{1+r/2}
\]
give
\begin{align}
&\frac{\mathrm d}{\mathrm du}\int_0^1
\ip{A^{(1+r)/2}}{(\Gamma_{v,B+uI\mid A}+zI)^{-1}(A^{(1+r)/2})}\dd v
\notag\\
&\qquad=-\int_0^1
\ip{A^{1+r/2}}{\mathcal L_v(B+uI+zA)^{-2}(A^{1+r/2})}\dd v.
\label{eq:inverse-test}
\end{align}
For $X>0$, diagonalization and
$\int_0^1((1-v)x_i+vx_j)^{-2}\dd v=(x_ix_j)^{-1}$ imply
\[
\int_0^1\mathcal L_v(X)^{-2}(H)\dd v=X^{-1}HX^{-1}.
\]
Thus \eqref{eq:inverse-test} equals
\[
-\Tr\!\left[A^{1+r/2}(B+uI+zA)^{-1}
A^{1+r/2}(B+uI+zA)^{-1}\right].
\]
Substitute
$(B+uI+zA)^{-1}
=A^{-1/2}(A^{-1/2}(B+uI)A^{-1/2}+zI)^{-1}A^{-1/2}$.
Trace cyclicity and $\D f[X](H)=-(X+zI)^{-1}H(X+zI)^{-1}$
identify this expression with the right side of
\eqref{eq:derivative-identity}.

For $z>\max J$, the expansion
$(x+z)^{-1}=\sum_{k\ge0}(-1)^kx^kz^{-k-1}$ converges in $C^1(J)$:
the series and its $x$ derivative converge absolutely and uniformly.
Continuity permits applying either functional termwise; uniqueness of
power series in $z^{-1}$ gives equality for monomials and polynomials.
Integrating uniform polynomial approximations to $f'$ and matching
one value gives polynomials $p_m\to f$ in $C^1(J)$, so continuity passes
the identity to $f$. Since $u_0$ was arbitrary, this proves
Lemma~\ref{lem:identity}, including the right derivative at zero.

\section{Proof of the power-derivative estimate}\label{app:derivative}

For $X,Y>0$, the weighted geometric mean is
\[
X\#_tY=X^{1/2}(X^{-1/2}YX^{-1/2})^tX^{1/2}
=Y\#_{1-t}X,\qquad 0\le t\le1;
\]
see \cite{Bhatia}. We write $X\preceq_{\log}Y$ when the products of
the largest $k$ eigenvalues satisfy the corresponding inequalities
for every $k$, with equality for the product of all eigenvalues.
Log-majorization is preserved by positive powers and implies
$\Tr X\le\Tr Y$.

We use the following three results in precise forms. Epstein's
concavity theorem \cite{Epstein}, also stated as
\cite[Corollary~1.10]{CarlenZhang}, says that, for fixed $K$ and $0<s<1$,
\begin{equation}\label{eq:epstein-functional}
	X\mapsto\Tr[(K^*X^sK)^{1/s}]
\end{equation}
is concave. Ando--Hiai log-majorization \cite{AndoHiai} gives
\begin{equation}\label{eq:ando-hiai}
	X\#_tY\preceq_{\log}(X^s\#_tY^s)^{1/s},
	\qquad 0<s<1,\quad 0\le t\le1.
\end{equation}
Finally, the reversed generalized Bebiano--Lemos--Provid\^encia
inequality, in \cite{BLP}, \cite[(14)]{LiuCheng}, yields
\begin{equation}\label{eq:reversed-gblp}
	X^{-1}(X^{1/2}YX^{1/2})^{1-s}X^{-1}
	\preceq_{\log}X^{-(1+s)/2}Y^{1-s}X^{-(1+s)/2}.
\end{equation}
Here the parameters are $(p,q,r)=(1,1-s,-2)$, which satisfy
$0\le q<p$ and $-r\ge q$. We have written the positive symmetrized
representatives of the products appearing in that reference; these
have the same eigenvalues as the corresponding nonsymmetric products.

\begin{proof}[Proof of Lemma~\ref{lem:derivative}]
	Fix $A,B>0$ and choose
	\[
	K=B^{-s/2}\left(B^{s/2}A^{s/2}B^{s/2}\right)^{s/2}.
	\]
	Then
	\begin{equation}\label{eq:K-identities}
		KK^*=B^{-s}\#_sA^{s/2},\qquad
		K(K^*B^sK)^{1/s-1}K^*=A^{s/2}.
	\end{equation}
	Keep this $K$ fixed when differentiating the concave, degree-one
	homogeneous map in \eqref{eq:epstein-functional}. Its supporting-hyperplane
	inequality at $B$ bounds its value at $A^{s/2}$ by its derivative at
	$B$ in direction $A^{s/2}$. The trace chain rule and
	\eqref{eq:K-identities} therefore yield
	\begin{align}
		\frac1s\Tr[A^{s/2}\D_s[B](A^{s/2})]
		&\ge\Tr[(K^*A^{s^2/2}K)^{1/s}]\notag\\
		&=\Tr\!\left[
		\left(A^{s^2/4}(B^{-s}\#_sA^{s/2})A^{s^2/4}\right)^{1/s}
		\right].
		\label{eq:epstein}
	\end{align}
	The last equality uses the equality of the eigenvalues of
	$K^*A^{s^2/2}K$ and $A^{s^2/4}KK^*A^{s^2/4}$.
	
	Next, \eqref{eq:ando-hiai} and \eqref{eq:reversed-gblp} imply
	\begin{align}
		A\#_{1-s}B^{-1}
		&\preceq_{\log}(A^s\#_{1-s}B^{-s})^{1/s}\notag\\
		&\preceq_{\log}
		\left(A^{s^2/4}(B^{-s}\#_sA^{s/2})A^{s^2/4}\right)^{1/s}.
		\label{eq:mean-chain}
	\end{align}
	To verify the second step, substitute
	$X=A^{-s/2}$ and $Y=A^{-s/4}B^{-s}A^{-s/4}$ into
	\eqref{eq:reversed-gblp}. Its left side is $A^s\#_{1-s}B^{-s}$;
	its right side is
	\[
	A^{s(1+s)/4}(A^{-s/4}B^{-s}A^{-s/4})^{1-s}A^{s(1+s)/4}
	=A^{s^2/4}(B^{-s}\#_sA^{s/2})A^{s^2/4}.
	\]
	Taking the power $1/s$ preserves log-majorization. Taking traces in
	\eqref{eq:mean-chain} and using \eqref{eq:epstein} gives
	\[
	\Tr\!\left[A(A^{1/2}BA^{1/2})^{s-1}\right]
	=\Tr[A\#_{1-s}B^{-1}]
	\le\frac1s\Tr[A^{s/2}\D_s[B](A^{s/2})],
	\]
	which is \eqref{eq:derivative}.
\end{proof}

%\clearpage

\end{document}